\documentclass[12pt]{amsart}
\usepackage{amsmath, amscd, mathrsfs}
\usepackage{amsthm}
\usepackage{verbatim}
\usepackage{amssymb}
\usepackage[all, cmtip]{xy}
\usepackage[colorlinks=true,pagebackref]{hyperref}
\usepackage[dvipsnames,usenames]{color}
\hypersetup{linkcolor=RawSienna,anchorcolor=BurntOrange,citecolor=OliveGreen,filecolor=BlueViolet,menucolor=Yellow,urlcolor=OliveGreen}

\newtheorem{lemma}{Lemma}[section]
\newtheorem{theorem}[lemma]{Theorem}
\newtheorem{corollary}[lemma]{Corollary}
\newtheorem{proposition}[lemma]{Proposition}

\newtheorem{conjecture}[lemma]{Conjecture}

\theoremstyle{definition}
\newtheorem{definition}[lemma]{Definition}
\newtheorem{remark}[lemma]{Remark}

\newtheorem{example}[lemma]{Example}

\theoremstyle{remark}
\newtheorem*{remark*}{Remark}
\newtheorem*{note*}{Note}

\newcommand{\CH}{\operatorname{CH}}

\newcommand{\Lie}{\operatorname{Lie}}
\newcommand{\eu}{\operatorname{eu}}
\newcommand{\GL}{\operatorname{GL}}

\newcommand{\codim}{\operatorname{codim}}
\newcommand{\Spec}{\operatorname{Spec}}

\newcommand{\BR}{\operatorname{BR}}

\newcommand{\bU}{{\mathbf U}}
\newcommand{\bW}{{\mathbf W}}
\newcommand{\bV}{{\mathbf V}}
\newcommand{\bY}{{\mathbf Y}}
\newcommand{\bX}{{\mathbf X}}

\newcommand{\G}{\mathbb{G}}

\newcommand{\CC}{\mathbb{C}}

\newcommand{\ix}{\mathcal {X}}
\newcommand{\iv}{\mathcal {V}}
\newcommand{\iu}{\mathcal {U}}
\newcommand{\iz}{\mathcal {Z}}
\newcommand{\iy}{\mathcal {Y}}
\newcommand{\bx}{\mathbf {X}}
\newcommand{\bv}{\mathbf {V}}
\newcommand{\bu}{\mathbf {U}}
\newcommand{\bz}{\mathbf {Z}}
\newcommand{\by}{\mathbf {Y}}
\newcommand{\bm}{\mathbf {M}}
\newcommand{\bn}{\mathbf {N}}

\newcommand{\ZZ}{\mathbb{Z}}

\newcommand{\QQ}{\mathbb{Q}}

\newcommand{\A}{\mathbb{A}}

\newcommand{\Pro}{\mathbb{P}}

\newcommand{\spec}{{\operatorname{Spec}}}

\newcommand{\Sym}{\operatorname{Sym}}

\newcommand{\Aut}{\operatorname{Aut}}
\newcommand{\Isom}{\operatorname{Isom}}
\newcommand{\Stab}{\operatorname{Stab}}

\newcounter{item-counter}

\begin{document}
\title[Strong regular embeddings and hypertoric geometry]{Strong regular embeddings of Deligne-Mumford stacks and
  hypertoric geometry}

\author{Dan Edidin}
\address{Department of Mathematics, University of Missouri-Columbia, Columbia, Missouri 65211}
\email{edidind@missouri.edu}
\thanks{The author's research was supported in part by a Simons collaboration grant.}

\begin{abstract}
We introduce the notion of {\em strong regular embeddings} 
of Deligne-Mumford stacks. These morphisms naturally arise in the related contexts of generalized Euler sequences and hypertoric geometry.
 \end{abstract}

\maketitle


\section{Introduction}
Let $G$ be a finite group acting on affine schemes $X = \Spec A$
and $Y = \Spec B$ and let $f \colon Y \to X$ be a $G$-equivariant morphism. 
Since $f$ is $G$-equivariant there is an induced map
of invariant subrings $A^G \to B^G$ corresponding to a  morphism of quotients
$g \colon Y/G \to X/G$. Certain algebro-geometric
properties of the morphism $f$ (typically related to finiteness)
are automatically preserved by the morphism $g$. For example, if $f$
is finite then the induced morphism of quotients $g \colon Y/G \to X/G$
is also finite. Likewise if $|G|$ is a unit in $\Spec B$ 
and $f \colon Y \to X$ is a closed embedding then $g \colon Y/G \to X/G$
is as well.

On the other hand, many properties of the morphism $f$ will not
descend. If $f$ is flat or smooth the induced morphism of
quotients need not be.  Instead we can impose additional
conditions on the actions of $G$ on $X$ and $Y$ to ensure that a
property of morphisms of schemes does descend to the quotient. Two
obvious conditions that suffice are that $G$ act freely or that
$G$ act trivially on both spaces.

Note, however, that these conditions are not necessary. For example
if $Y = X \times Z$ and $G$ acts trivially on $Z$ then 
$Y/G = X/G \times Z$ so the flat projection $Y \to X$ descends
to a flat projection $Y/G \to X/G$
and the diagram
$$\xymatrix{Y \ar[r] \ar[d] & \ar[d] X\\Y/G \ar[r] & X/G}$$
is cartesian. 

This is an example of a {\em stabilizer preserving morphism}, meaning
that for every point $y \in Y$ the map of stabilizers $\Stab_y Y \to
\Stab_{f(y)} X$ is an isomorphism of groups. A well known folklore
theorem states that an \'etale stabilizer preserving morphism is {\em
  strongly \'etale}.  This means that the morphism of quotients $Y/G
\to X/G$ and the diagram
$$\xymatrix{Y \ar[r] \ar[d] & \ar[d] X  \\Y/G \ar[r] & X/G}$$
is cartesian.

In this paper we turn our attention to the problem of determining when
a regular embedding $i \colon Y \hookrightarrow X$ descends to a regular embedding of quotients $j \colon Y/G \hookrightarrow  X/G$ (necessarily of codimension
$d$ since $\dim Y = \dim Y/G$ and $\dim X = \dim X/G$). 
The $G$-equivariant morphism $Y \to X$ is a local model for a morphism
$\iota \colon \iy \to \ix$ of Deligne-Mumford stacks. We say that a morphism of Deligne-Mumford stacks is a {\em strong regular embedding} if the induced
morphism of coarse spaces $Y \to X$ is a regular embedding and
the diagram of stacks and spaces 
$$\xymatrix{ \ar[r] \ar[d] \iy & \ix \ar[d]\\X \ar[r] & Y}$$
is cartesian.

Although any immersion of stacks is stabilizer preserving, not every regular
embedding is strong (Example \ref{ex.strong}). 
Our first result, Theorem \ref{thm.strongembedd}, is
a characterization of strong regular embeddings. When $\iy \hookrightarrow
\ix$ is a strong regular embedding of smooth stacks we use Artin's approximation theorem \cite{Art:69} to prove that the induced morphism 
of coarse spaces $Y \hookrightarrow X$ is a \'etale locally the section
of a smooth morphism. As a corollary we prove that if $\tilde{X} \to X$
is a functorial resolution of singularities (in the sense of 
\cite{Kol:07}) then the fiber product $\tilde{Y} = Y \times_X \tilde{X}$
is a resolution of singularities of $Y$. 

We also prove that the pullback on Chow groups $(I\iota)^*
\colon \CH^*(I\ix)\to \CH^*(I\iy)$ commutes with the Chen-Ruan orbifold product
thereby giving a ring homomorphism of orbifold Chow rings 
$\CH^*_{orb}(\ix) \to \CH^*_{orb}(\iy)$.

\subsection{Applications}
Strong regular embeddings arise in two related contexts, generalized
Euler sequences and hypertoric geometry.

\subsubsection{Generalized Euler sequences}
If $\ix = [X/T]$ is a quotient stack with $X$ a smooth scheme defined over a field
$k$ 
and $T$ is a torus acting with finite stabilizer
then $\ix$ is a smooth Deligne-Mumford stack whose coarse space is a scheme. 
The cotangent bundle $T^*\ix$
fits into a generalized Euler sequence of vector bundles on $\ix$
$$ 0 \to T^*\ix \to [T^*X/T] \to \ix \times \Lie(T)^* \to 0$$
which implies by Theorem \ref{thm.strongembedd} that the inclusion
$T^*\ix \to [T^*X/T]$ is a strong regular embedding. We prove that if
$\bm$ and $\bn$ are the coarse spaces of $T^*\ix$ and $[T^*X/T]$
respectively then the regular embedding $i \colon \bm \to \bn$ induces an
isomorphism of integral Chow groups $i^* \colon \CH^*(\bn) \to
\CH^*(\bm)$.  We also conjecture that if $\tilde{\bn}\to \bn$ is a
functorial resolution of singularities, pullback along the
morphism $\tilde{\bm} \to \tilde{\bn}$ obtained by base change induces an
isomorphism of integral Chow rings $\CH^*(\tilde{\bn}) \to
\CH^*(\tilde{\bn})$. Finally we prove that the pullback $(I\iota)^*
\colon \CH^*(I[T^*X/T]) \to \CH^*(IT^*\ix)$ induces an isomorphism of
orbifold Chow rings.

\subsubsection{Hypertoric stacks}
Hypertoric varieties were first defined by Bielawski and Dancer
\cite{BiDa:00} and give a large and interesting class of algebraic
symplectic manifolds with complete hyperKahler metrics. Algebraic
symplectic manifolds naturally occur in a variety of mathematical
contexts including representation theory (Springer resolutions,
Nakajima quiver varieties, Slodowy slices), string theory (moduli spaces of Higgs bundles) and mirror
symmetry (Gromov-Witten theory of K3 surfaces and $T^*{\mathbb P}^1$ ).

The theory was further developed by a number of authors including
Hausel and Sturmfels (\cite{HaSt:02}) and Proudfoot
(\cite{Pro:08,Pro:11}.

Let $T$ be a torus of rank $d$ and let $V$ be an $n$-dimensional representation of
$T$ with $d \times n$ weight matrix $A = (a_{ij})$. 
There is a natural algebraic moment map $\mu \colon 
V \times V^* \to \Lie(T)^*$
and the hypertoric variety ${\mathbf Y}(A,\theta)$
is defined as the double reduction $(V \times V^*)////T$
where the first reduction is an algebraic symplectic reduction with respect
to $\mu$ and the second reduction is a GIT quotient with respect to the torus
action linearized with respect to a character $\theta \in X(T)$. 

The hypertoric variety ${\mathbf Y}(A,\theta)$ naturally
embedds in a {\em Lawrence toric variety}
${\mathbf X}(A^\pm, \theta)$
which is the GIT quotient $(V \times V^*)//T$. Following Jiang and Tseng
\cite{JiTs:08a} we refer to the corresponding quotient stack
$\iy(A,\theta)$ (resp. $\ix(A^\pm, \theta)$)
as a {\em hypertoric stack} (resp. {\em Lawrence toric stack}). 
If the character $\theta$ is generic then the stacks $\iy(A,\theta)$
and $\ix(A^\pm, \theta)$ are Deligne-Mumford and the 
corresponding varieties are their coarse spaces.

In Section \ref{sec.hyperkahler} we prove that the embedding
$\iy(A,\theta) \hookrightarrow \ix(A^\pm,\theta)$ is a strong regular
embedding. This implies
that the morphism ${\mathbf Y}(A,\theta) \to {\mathbf X}(A^\pm,\theta)$
is a regular embedding which
is \'etale locally a section of smooth morphism, so
that any functorial resolution of singularities of the Lawrence 
toric stack ${\mathbf X}(A^\pm,\theta)$
pulls back to resolution of singularities of the hypertoric variety
${\mathbf Y}(A,\theta)$. 

We also prove (Theorem \ref{thm.chowhypertoric}) that the inclusion morphism $\iota \colon \iy(A,\theta)
\hookrightarrow \ix(A^\pm, \theta)$ induces an isomorphism of
integral Chow rings $\CH^*(\ix(A^\pm,\theta) \to \CH^*(\ix(A,\theta))$
and integral Chow groups $\CH^*({\mathbf X}(A^\pm,\theta)) \to
\CH^*({\mathbf Y}(A,\theta))$.

As a corollary of Theorems \ref{thm.chowhypertoric} and \ref{thm.orbifoldproducts} we also prove
(Corollary \ref{cor.hypertoricorbifold}) that the inclusion of inertia stacks
$I\iy(A,\theta) \to I\ix(A^\pm, \theta)$ induces an isomorphism
of integral orbifold Chow rings. With rational coefficients
this result was previously obtained by Jiang and Tseng \cite{JiTs:08a}.

\subsection{Conventions and notation}

\subsubsection{Generalities on stacks} 
Unless otherwise stated we work with Deligne-Mumford stacks of finite
type (and hence finite presentation) over a Noetherian scheme
$S$.

Let $\ix$ be a Deligne-Mumford stack.
If $T$ is an $S$-scheme and $T \stackrel{x} \to \ix$ is a $T$-valued
point then we use the notation $\Aut(x)$ to denote the $T$-group
$\Isom_T(x,x)$. Following Abramovich and Vistoli \cite{AbVi:02}
we say that a Deligne-Mumford stack is {\em tame}
if for every geometric point $\spec \;k \stackrel{x} \to \ix$ 
of $\ix$, the finite group $\Aut(x)$ is linearly reductive over $\spec\; k$.
This is equivalent to saying that $|\Aut(x)|$ is prime to the characteristic 
of $k$. 

If $\ix$ is a Deligne-Mumford stack we let $I\ix$ denote the {\em inertia
stack} $\ix \times_{\ix \times \ix} \ix$ where the fiber product is
taken with respect to the diagonal $\ix \to 
\ix \times \ix$. 
Because $\ix$ is Deligne-Mumford the natural morphism $I\ix \to \ix$
is unramified. The fiber of $I\ix \to \ix$
over a $T$-valued point $T \stackrel{x} \to \ix$
is the  $T$-group $\Aut(x)$. We denote by $I^2\ix$
the {\em double inertia} stack defined as the fiber product $I\ix \times_\ix I\ix$. Note that $I^2\ix$ is not in general equivalent to the stack
$I(I\ix)$. 

A Deligne-Mumford stack $\ix$ has {\em finite stabilizer} if
the morphism $I\ix \to \ix$ is finite. The stack
is separated if the diagonal $\ix \to \ix \times \ix$ is finite.
A fundamental theorem 
of Keel and Mori implies that any Deligne-Mumford stack with finite stabilizer
has a coarse space $X$.

An algebraic space $X$ is a {\em coarse space} of a stack $\ix$
if there is a morphism $\ix \stackrel{\pi} \to X$ such that \\
$(i)$ $\pi$ is initial for maps from $\ix$ to algebraic spaces;
ie for any morphism $f \colon \ix \to Z$ with $Z$ an algebraic space
there is a unique morphism of algebraic spaces $g \colon X \to Z$
such that the morphism $f$ factors as $g \circ \pi$

$(ii)$ $\pi$ induces a bijection between geometric points of $\ix$
and geometric points of $X$.

The universal property $(i)$ implies that the coarse space is unique
up to (unique) isomorphism of algebraic spaces.

A stack $\ix$ is a {\em quotient stack} if it is equivalent to a stack
of the form $[X/G]$ where $X$ is an algebraic space and $G$ is flat
closed subgroup scheme of $\GL_n(S)$ for some $n$. Equivalently we may assume
$\ix = [X'/\GL_n(S)]$ where $X'$ is the algebraic space $(X \times \GL_n(S))/G$.

If $\ix=[X/G]$ is a quotient stack then the coarse space of $\bX$ 
of $\ix$ is the geometric quotient $X/G$ in the category of algebraic spaces.

\subsubsection{Chow groups}
If $X$ is a scheme defined over a field we denote by
$\CH_d(X)$ the group of $d$-dimensional cycles modulo
rational equivalence. Likewise we denote by $\CH^d(X)$ the group
of codimension-$d$ cycles modulo rational equivalence. When
$X$ is equidimensional of dimension $n$,  $\CH_d(X) = \CH^{n-d}(X)$.
Set $CH_*(X) = \oplus_d CH_d(X)$ and $CH^*(X) = \oplus_d CH^d(X)$.
If $f \colon Y \to X$ is proper then there is a pushforward 
$f_* \colon CH_d(Y) \to CH_d(X)$. Likewise if $f$ is flat, or an lci morphism
there is a pullback $f^*\colon CH^d(X) \to CH^d(Y)$. When
$X$ is smooth, the diagonal is a regular embedding and pullback
along the diagonal defines a graded ring structure on $CH^*(X)$.

This definition of Chow groups formally extends to algebraic spaces.

If $G$ is a linear algebraic group acting on an algebraic space $X$
then we can define equivariant Chow groups \cite{EdGr:98}.  The group
$\CH^d_G(X)$ is defined as $\CH^d((X \times U)/G$ where $U$ is an open
set in a representation $V$ of $G$ such that $G$ acts freely on $U$
and $\codim (V \smallsetminus U) > d$. This definition is independent
of the representation $V$ and the open set $U$.

If $\ix$ is a Deligne-Mumford stack then $\CH^d(\ix)$ denotes the Chow group
in codimension $d$ defined by Kresch in \cite{Kre:99a}. Typically we will work with
quotient stacks $\ix = [X/G]$. In this case $\CH^d(\ix) = CH^d_G(X)$
where $CH^d_G(X)$ denotes the equivariant Chow group defined above.
Note that $CH^d(\ix)$ can be non-zero for $d > \dim \ix$, although
this group will be torsion. Likewise the group $CH_d(\ix)$ can be non-zero
for negative $d$.

Bloch defined for a scheme $X$, {\em higher Chow groups} $\CH_d(X,k)$
parametrizing cycles of dimension $d+k$ on $X \times \Delta^k$ which
intersect the faces of $\Delta^k$ transversely. (Here $\Delta^k$
denotes the algebraic $k$-simplex $\Spec k[t_0, \ldots , t_k]/(t_0 +
\ldots + t_k - 1)$.) In this definition $CH^d(X,0) = CH^d(X)$.

A key property of higher Chow groups proved by Bloch \cite{Blo:86}
and Levine \cite{Lev:01} is that if $X$ is an arbitrary scheme
and $Z \subset X$ is a closed subscheme with complement $U$
there is a localization long exact sequence
\begin{equation} \label{eq.localization}
\ldots \to CH_d(Z,k) \to CH_d(X,k) \to CH_d(U,k)\to \ldots
\end{equation}
extending the classical localization sequence
$$CH_d(Z) \to CH_d(X) \to CH_d(U) \to 0$$
for ordinary Chow groups.
\subsection{Acknowledgments}
The author is grateful to Daniel Bergh and Daniel Lowengrub for helpful discussions.

\section{Definitions and statements of results}
\begin{definition} \label{def.strongreg} A morphism of Deligne-Mumford 
stacks $\iota
  \colon \iy \to \ix$ is said to be a {\em strong regular embedding}
 if the induced morphism of coarse spaces $i 
  \colon Y \to X$ is a regular embedding and the diagram
$$\begin{array}{ccc}
\iy & \stackrel{\iota} \to & \ix\\
\downarrow & & \downarrow\\
Y & \stackrel{i} \to & X
\end{array}$$
is cartesian.
\end{definition}

Our first result is the following characterization of strong regular embeddings:
\begin{theorem} \label{thm.strongembedd}
Let $\iota \colon \iy \to \ix$ be a regular embedding of tame Deligne-Mumford
stacks and let $Y \to X$ be the corresponding embedding of coarse spaces.  The following conditions
are equivalent:\\

(i) $\iota \colon \iy \to \ix$ is a strong regular embedding.

(ii) There is a stratification of $\iy$ by locally closed substacks
such that the normal bundle $N_\iota$ restricted to each stratum is trivial.

(ii') If $Y$ is a scheme then $N_\iota$ is locally trivial in the Zariski topology on $\iy$.

(iii) For every geometric point $\Spec k \stackrel{y} \to \iy$ the pullback
of $(N_\iota)_y$  is a trivial representation of
the inertia group $G_y = \Aut(y)$.

(iv) For every geometric point $\Spec k \stackrel{y} \to \iy$
there is a strongly \'etale morphism $[U/G_y] \to \ix$ with $U$ affine such
that normal bundle to the regular embedding 
$\iy \times_\ix [U/G] \hookrightarrow [U/G]$ is trivial.

(v) 
For every geometric point $\Spec k \stackrel{y} \to \iy$
there is a strongly \'etale morphism $[U/G_y] \to \ix$ where $U$ is an affine scheme such that $\iy \times_\ix U$
is defined by a $G$-fixed regular sequence in ${\mathcal O}(U)$.
\end{theorem}

\subsection{Strong regular embeddings and resolutions of singularities}
Given a strong regular embedding $\iota \colon \iy \to \ix$ of smooth
Deligne-Mumford stacks we obtain a regular embedding of the possibly singular
coarse spaces $Y \to X$. 
Our next result shows that this morphism is rather distinguished among regular embeddings in that it is \'etale locally a section of a smooth morphism.
\begin{theorem} \label{thm.localsection}
Let $\iy \to \ix$ be a strong regular embedding of smooth Deligne-Mumford stacks
and let $Y \to X$ be the induced regular embedding of coarse spaces. Then
for any point $y$ of $Y$ there are \'etale neighborhoods $W \to Y$ of $y$
in $Y$ and $Z \to X$ of $y$ in $X$
such that the morphism $W \to X$ factors as $W \to W' \to Z \to X$
where $W'\to Z$ 
a section of a smooth morphism $Z \to W'$ and the morphism $W \to W'$ is
\'etale.
\end{theorem}

As an application of Theorem \ref{thm.localsection}
we can show
that the induced morphisms of coarse moduli spaces has a strong functorial property with respect to resolutions of singularities.
\begin{corollary} \label{cor.resolution} Let $\iota \colon \iy \to
  \ix$ be a strong regular embedding of smooth, tame Deligne-Mumford 
stacks defined over a perfect field $k$.  Then,
  if $\tilde{X}$ is a canonical (functorial) resolution of
  singularities of $X$ then $Y \times_X \tilde{Y}$ is a a resolution
  of singularities of $Y$.
\end{corollary}
\begin{remark} To date, the existence of resolutions of
  singularities has only been established in characteristic 0.
\end{remark}

\subsection{Strong regular embeddings and orbifold products}
For stacks defined over a field we also obtain the following result about orbifold products. See
Section \ref{sec.orbifoldproducts} for the notation. The proof will be given
in Section \ref{sec.orbifoldproducts}.

\begin{theorem} \label{thm.orbifoldproducts}
 Let $\iota \colon \iy \to \ix$ be a strong regular embedding of
  smooth Deligne-Mumford stacks and let $I\iota \colon I\iy \to I\ix$
be the induced morphism of inertia stacks.  
The pullback $(I\iota)^* \colon \CH^*(I\ix) \to
  \CH^*(I\iy)$ commutes with the orbifold products on $\CH^*(I\ix)$ and
  $\CH^*(I\iy)$. 
\end{theorem}

As an application we obtain the following corollaries which will be proved in 
Section \ref{sec.cotangent} and Section \ref{sec.hyperkahler} respectively.
\begin{corollary} \label{cor.cotangent}
Let $T$ be a torus acting with finite stabilizer on
a smooth variety $X$ and let $\ix = [X/T]$ be the quotient stack.
 The pullback
$(I\iota)^* \colon \CH^*(I [T^*X/T]) \to \CH^*(I T^*\ix)$ 
induces an isomorphism of orbifold
Chow rings $\CH^*_{orb}([T^*X/T]) \to \CH^*_{orb}(T^*\ix)$.
\end{corollary}

\begin{corollary} \label{cor.hypertoricorbifold} If $\iy(A, \theta)$
  is a hypertoric stack\footnote{See Section \ref{sec.hyperkahler}
    for notation and definitions.} and $\ix(A, \theta)$ is the
  associated Lawrence toric stack then the pullback on Chow groups
  $(I\iota)^* \colon \CH^*(\ix(A, \theta)) \to \CH^*(\iy(A, \theta))$
  induces an isomorphism of orbifold Chow rings.

The analogous results also hold for orbifold $K$-theory.
\end{corollary}
\begin{remark}
With rational coefficients, Corollary \ref{cor.hypertoricorbifold} was proved by direct calculation by Jiang and Tseng in \cite[Theorem 3.10]{JiTs:08a}.
\end{remark}


%


\section{Proofs of Theorems} 

\subsection{Proof of Theorem \ref{thm.strongembedd}}

\paragraph{(i)$ \implies $(ii), (ii')} Since the relative dimension of
the morphism $\iy \stackrel{\iota} \to \ix$ is necessarily the same 
as the relative dimension 
of the morphism $Y \stackrel{i} \to X$, the normal
bundle $N_\iota$ is the pullback of the normal bundle $N_i$.
If $Y$ is an algebraic
space then it has a stratification by schemes, and on each open
stratum $N_i$ is locally trivial. Refining the stratification if
necessary we obtain one where $N_i$ is trivial on each open
stratum. Pulling back to $\iy$ gives the desired conclusion about
$N_\iota$. In the case where $Y$ is a scheme then $N_i$ is locally
trivial in the Zariski topology and (ii') follows.

\paragraph{(ii) $\implies$ (iii)}
Let $\spec k \stackrel{y} \to \iy$ be a geometric point of $\iy$.  
Our hypothesis (ii) implies that the pullback of
$N_\iota$ to $BG_y$ is trivial. This means that the fiber of $N_\iota$
at $y$ is a trivial $G_y$-module which is assertion (iii).

\paragraph{(iii) $\implies$ (iv)}
By the local structure theorem for Deligne-Mumford stacks (cf. \cite[Lemma 2.2.3]{AbVi:02}), 
 every geometric point
$y \to \iy$ has a strongly \'etale neighborhood isomorphic 
to $[U/G_y]$ where $U$ is an affine scheme
and $G_y = \Aut_y(\iy)$. The pullback $N$ of $N_\iota$
to $[U/G_y]$ corresponds to a $G_y$-equivariant vector bundle $N$ on the affine
scheme $U$. Let $O$ be a point of $U$  mapping to $y$ 
so that $O$ is fixed by $G_y$. In a neighborhood of the fixed point $O$
the bundle
$N$ is non-equivariantly trivial. Taking the intersection over all $g \in G_y$
of the translates of this neighborhood we obtain a $G_y$-invariant  neighborhood of $O$ 
such that the restriction of $N$ to this neighborhood is non-equivariantly trivial.
Replacing $U$ with this neighborhood we may assume that $N$ decomposes
as $\oplus_V V \otimes {\mathcal O}_U^{r_V}$ where the sum is over the irreducible representations of $G_y$. 
By hypothesis the fiber of $N$ at the fixed point
$O$ is a trivial representation. So $N$ must be globally trivial as a $G_y$-equivariant bundle. 

\paragraph{(iv) $\implies $(v)}
Again let $O$ be a fixed point for the $G_y$ action which maps to $y$ and let
$V = \iy \times_\ix U$. Then $\iy \times_\ix [U/G_y] = [V/G_y]$.
Since $O$ is a fixed point the local ring $A = {\mathcal O}_{U,O}$
has an induced action of $G_y$. Let $B$ denote the local ring
${\mathcal O}_{V,G}$ so 
$B=A/I$. We wish to show that $I$ is generated by a $G_y$-fixed regular sequence.
By assumption $G_y$
acts trivially on the $B$-module $I/I^2$. Then
$I/I^2$ is free with basis $\overline{x}_1, \ldots
\overline{x}_r$ each of which is $G_y$-fixed. 
Let
  $x_1, \ldots , x_m$ be lifts of these elements to a regular sequence
  in $A$. By construction $gx_i \equiv x_i \mod I^2$. Since $G$ is
  linearly reductive the $G$-module has a Reynolds operator
  $\rho$. Let $y_i = \rho(x_i)$ then $y_i \equiv x_i \mod I^2$ so the
  $y_i$ form a $G_y$-fixed regular sequence that generate $I$.

\paragraph{(v) $\implies$ (i)}
Since $[U/G] \to \ix$ is strongly \'etale the corresponding
morphism $U/G \to X$ is \'etale, and to check that $\iota \colon \iy \to \ix$
is a strong regular embedding it suffices to check that
the morphism $[V/G] \to [U/G]$ is a strong regular embedding where 
$V = \iy \times_\ix U$. 

As above assume that $U = \Spec A$ and $V = \Spec A/I$. Then we must show
that the
kernel of the morphism of invariant rings
$A^G \to (A/I)^G$ is generated by a regular sequence and that $A/I =
A^G \otimes_{A^G}(A/I)^G$.

Let $x_1, \ldots , x_r\in A$ be a regular sequence of $G$-fixed elements that generate $I$. Then these elements also generate the invariant ideal $I^G$. Since $G$ is linearly reductive
$(A/I)^G = A^G/I^G$ as $A^G$-modules. Also, since $A^G$ is a subring of $A$ the sequence $x_1, \ldots , x_n$ is also regular in $A^G$. Finally, since $I$ is generated by $G$-fixed elements
$I^GA = I$ so $A/I = A \otimes_{A^G} (A/I)^G$ as claimed.

\begin{example} \label{ex.strong} We illustrate the failure of the Theorem when $\iota$ is
  not a strong regular embedding. Let $Y = \A^2$ with the action of $G
  = \pm 1$ given by $(-1)(x,y) = (-x,-y)$ and let $X =
  \{(0,0)\}$. Then the morphism of quotient stacks $BG = [X/G] \to
  [Y/G]$ is a regular embedding. However the induced morphism of
  coarse moduli spaces is the inclusion of the singular point in the
  affine quadric cone $\Spec k[x^2,y^2,xy]$.

In this example $I$ is not generated by a $G$-fixed regular sequence and $I^G$ is generated by 
$(x^2,xy,y^2)$ which do not form a regular sequence in the invariant
subring $k[x^2,xy,y^2]$.
The normal
  bundle to $\iy$ in $\ix$ is the vector bundle on $BG$ corresponding to the
  non-trivial two-dimensional representation of $G$ with weights
  $(-1,-1)$.

On the other hand if we consider the action given by $(-1)(a,b) = (-a,b)$
and let $Y = Z(y)$ then $Y$ is defined by a $G$-fixed regular sequence
and $I = (y) = I^G$ so that $A/I = \Spec k[x]$ is obtained by extension of scalars from $A^G/I^G = \Spec k[x^2]$.
\end{example}

\begin{example} The tameness assumption is crucial so that the group $G$
is linearly reductive. Let $k$ be field of characteristic 2 and let
$G=\ZZ_2$ act on $k[x,y]$ be exchanging coordinates. Let $I
=(x+y)$. Then $I$ is generated by a $G$-fixed regular sequence. 
However, the sequence of $G$-modules $0 \to I \to A \to A/I$ does not remain 
exact after taking $G$-invariants, since the map
$A^G = k[x+y, xy]  \to (A/I)^G = k[x,y]/(x+y)$ is not surjective. 
In particular the map $\Spec(A/I)^G \to \Spec A^G$ is not a closed embedding.
Hence the regular embedding of stacks $[(\Spec A/I)/G] \to [(\Spec A)/G]$
is not a strong regular embedding. 
\end{example}

\subsection{Proofs of Theorem \ref{thm.localsection} and Corollary
\ref{cor.resolution}}
\begin{proof}[Proof of Theorem \ref{thm.localsection}]
  Let $y \to \iy \hookrightarrow \ix$ be a geometric point. By the
  local structure theorem for tame Deligne-Mumford stacks there is a strongly
  \'etale neighborhood $[U/G_y,O] \to (\ix,y)$ where $U= \Spec A$ is
  affine and $O \in U$ is $G_y$-fixed.

Since the morphism $\iy \to \ix$ is representable and affine, the
inverse image $V$ of $\iy$ in $U$ is affine and the morphism
$([V/G_y],O) \to (\iy,y)$ is also strongly representable and we have a
cartesian square
$$\xymatrix{
([V/G_y],O) \ar@{^{(}->}[r]\ar[d] &  ([U/G_y],O)\ar[d]\\
(\iy,y) \ar@{^{(}->}[r] &  (\ix,y)
}$$
where the vertical maps are strongly \'etale.

Thus we are reduced to the case that $\ix = [U/G]$ with $U = \Spec A$
affine, $G$ a finite reductive group, and $\iy = [V/G]$ where $V =
\Spec A/I$ is a closed subscheme whose ideal $I=(f_1, \ldots , f_r)$ is
generated by a $G$-invariant regular sequence. Since $O \in U$ in
assumed to be $G$-fixed there is an induced action of $G$ on
$\hat{A}$, the completion of $A$ at the maximal ideal of $O$.  Since
$V$ is smooth over the ground scheme, 
there is an isomorphism of $G$-algebras $\hat{A} \simeq
\widehat{A/I}[[\overline{f}_1, \ldots \overline{f}_r]]$ where
the $\overline{f}_i$ is any lift of $f_i$ to $A$. Since $G$ acts
trivially on the $f_i$ we obtain an isomorphism $\hat{A}^G \simeq
(\widehat{A/I})^G[[\overline{f}_1, \ldots \overline{f}_r]]$.

Since $G$ is linearly reductive, $\hat{A}^G = \widehat{A^G}$ where the
completion of $A^G$ is taken at the contraction of the maximal ideal
corresponding to $O \in \Spec A$.
Thus the completion of the local ring of $X= U/G$ at the image of O is
isomorphic to a formal power series ring over the completion of the
local ring of $Y = V/G$ at the image $P$ of $O$ in $Y$. Applying
Artin's \'etale approximation theorem \cite[Corollary 2.6]{Art:69},
there is a scheme $W$ and \'etale morphisms $W \to X$ and $W \to Y
\times \A^r$ whose image contains $P$ and $P \times 0$ where $0$ is
the origin in $\A^r$.  Let $Z'= Y \times_{Y \times \A^r} Z$ so the
induced morphism $Z' \to W$ is a section of a smooth morphism.  Applying
Artin's approximation to $Z'$ and $Y$ yields \'etale morphisms $Z \to
Y$ and $Z \to Z'$ and a commutative diagram
$$\xymatrix{Z \ar[dr] \ar[r] & Z' \ar[d] \ar[r] & W \ar[d]\\
& Y \ar[r] & X}$$
\end{proof}

Corollary \ref{cor.resolution} follows from Theorem \ref{thm.localsection}
and the following proposition.

\begin{definition}
A regular embedding 
$Y \stackrel{i} \hookrightarrow X$
with the following local structure is called a {\em tubular} regular embedding
:\\
For each point of $y \in Y \subset X$ there are \'etale neighborhoods
  of $W\to X$ and $Z \to Y$ such that 
the morphism $Z \to X$ factors as $Z \to Z' \to W \to X$
where $Z' \to W$ is a section of a smooth morphism and $Z \to Z'$
is \'etale 
\end{definition}

\begin{proposition} \label{prop.resolution} Let $\BR$ be a resolution
  of singularities functor which is functorial for smooth morphisms
  \cite[Chapter 3]{Kol:07} and let $i \colon Y \to X$ be a tubular
  regular embedding. Let $\tilde{X} = BR(X)$. Then $i^*\tilde{X}$ is a
  resolution of singularities of $\tilde{X}$.  (Here $i^*\tilde{X}$
  refers to the fiber product $Y \times_X \tilde{X}$.)
\end{proposition}

\begin{proof}
First suppose that $i \colon Y \to X$ is a section of a smooth morphism
$\pi \colon X \to Y$. In this case, by functoriality, $\tilde{X} =\pi^*\tilde{Y}$ where $\tilde{Y} = BR(Y)$.
Since $\pi \circ i =1_Y$, $i^*\tilde{X} = \tilde{Y}$.

For the general case observe that if $Y \to X$ is a regular embedding
then the image of the smooth locus of $Y$ is contained in the smooth
locus of $X$. The reason is that any regular embedding with a smooth
source must have a smooth target by \cite[Theorem 17.12.1]{EGA4}.

Thus $i^*Y \to Y$ is an isomorphism over the smooth locus of $Y$. To
finish the proof we must show that $i^*Y$ is smooth. This can be done
after \'etale base change. By hypothesis we have a commutative diagram
$$\xymatrix{Z \ar[dr]^h \ar[r]^g & Z' \ar[d] \ar[r]^j & W \ar[d]^f\\
& Y \ar[r]^i & X}$$
where $j$ is a section of a smooth morphism and $f,g,h$ are \'etale.
Hence $\tilde{Z} = g^*j^*f^*\tilde{X}$.
On the other hand by commutativity $g^*j^*f^*\tilde{X} =h^*i^*\tilde{X}$.
Thus $i^*\tilde{X}$ is smooth after base change by the \'etale morphism
$h$.
\end{proof}

\begin{remark}
A result similar to Proposition \ref{prop.resolution} was proved by Daniel
Lowengrub in \cite{Low:15}.
\end{remark}

\subsection{Orbifold products and Proof of Theorem 
\ref{thm.orbifoldproducts}} \label{sec.orbifoldproducts}

\subsubsection{Definition of the obstruction class and orbifold product}
If $\ix$ is a smooth Deligne-Mumford stack with finite stabilizer then
Chen and Ruan \cite{ChRu:02} defined an exotic ring structure on the
cohomology of the inertia stack $I\ix$ called the orbifold
product. This product has been studied by many authors and extended to
both Chow groups and K-theory.  We briefly recall the definition
the orbifold product using the formalism developed in
\cite{JKK:07,EJK:10,EJK:12a}.

Denote by $I^2\ix$ the fiber product $I\ix \times_\ix I \ix$. 
Denote by $e_1$ and $e_2$ the two projections $I^2\ix \to \ix$.
Since $I^2\ix$ has the structure as a relative group scheme
over $I\ix$ there is an additional morphism $\mu \colon I^2\ix \to I\ix$
corresponding to the composition in this group.

The orbifold product on $\CH^*(I\ix)$ is defined as follows.
Given $\alpha, \beta \in \CH^*(I\ix)$,
$$
\alpha \star \beta= \mu_*\left(e_1^*\alpha \cdot e_2^*\beta \cdot 
\eu({\mathscr R_\ix}) \right)$$
where ${\mathcal R}_\ix$ is the {\em obstruction} bundle and $\eu$
denotes its top Chern class. (Note that $I\ix$ is in general not equidimensional.) This product preserves the {\em age grading} on the Chow groups of
$I\ix$. (See \cite{JKK:07} for the definition of the age grading.)

An analogous product can be defined in $K$-theory where
the symbol $\eu({\mathscr R}_X)$ refers to the $K$-theoretic Euler
class $\lambda_{-1}({\mathscr R}_X^*)$.

We denote by $\CH^*_{orb}(\ix)$ the group $\CH^*(I\ix)$ with the orbifold product
and age grading.

If we make the very mild assumption 
that $\ix = [X/G]$ with $G$ a linear algebraic group
acting with finite stabilizer on a smooth algebraic space $\ix$ then
the formalism of \cite{JKK:07, EJK:10, EJK:12a} can be used
to define the class of the obstruction bundle in $K_0(I^2\ix)$, the
Grothendieck group of vector bundles on $I^2\ix$. 
To do this we recall the definition of the logarithmic trace.

\begin{definition}\cite[Definition 4.3]{EJK:10}
Let $X$ be an algebraic space with the action
of an algebraic group $Z$ and let $V \to X$ be a $Z$-equivariant vector bundle
on $X$. Let $g$ be a finite order automorphism of order $r$
of the fibers of $V \to X$
such that the action of $g$ commutes with the action of $Z$.
Set $L(g)(V) = \sum_{k=1}^{r-1} {k\over{r}} V_k \in K_0(Z,X) \otimes \QQ$
where $V_k$ is the $e^{2\pi i k/r}$-eigenspaces for the action of
$g$. (Here $K_0(Z,X)$ denotes the Grothendieck group of $Z$-equivariant
vector bundles on $X$.)
\end{definition}
\begin{remark}
Observe that $L(g)(V) = 0$ if and only if $g$ acts trivially on the fibers
of $V \to X$.
\end{remark}

Under the assumption that $\ix = [X/G]$, $I^2\ix = [I^2_G X /G]$ where $I^2_GX =
  \{(g_1, g_2, x)| g_1 x = g_2 x\}$. Following \cite{EJK:10} $I^2_GX$
  decomposes into a disjoint sum indexed by double conjugacy
  \footnote{A double conjugacy class is the orbit of a pair $(g_1,
    g_2) \subset G \times G$ under the diagonal action of $G$ by
    conjugation.}  classes $\Psi \subset G \times G$. Specifically
  $I_GX = \coprod_\Phi I_\Phi$ where $$I_\Phi = \{(g_1,g_2,x)| g_1x =
  g_2 x = x, (g_1, g_2) \in \Phi\}.$$ If $(g_1, g_2) \in \Phi$ then
  $[I_\Phi/G] = [X^{g_1, g_2}/Z_G(g_1,g_2)]$ where $Z_G(g_1, g_2)$ is
  the subgroup of $G$ centralizing $g_1$ and $g_2$.

  On a component $[I_\Phi/G]= [X^{g_1,g_2}/Z_G(g_1,g_2)]$ the class of
  the obstruction bundle is given as an element of
  $K_0(Z_G(g_1,g_2),(X^{g_1,g_2})$ by the formula \cite[Definitions 2.2.3, 2.2.6, 2.3.3]{EJK:12a}
  \begin{equation} \label{eq.logtrace} LR({\mathbb T}) :=
L(g_1)({\mathbb T}|_{X^{g_1,g_2}}) + L(g_2)({\mathbb
        T}|_{X^{g_1,g_2}}) + L((g_1g_2)^{-1})({\mathbb T}|_{X^{g_1,g_2}}) -
        {\mathbb T}|_{X^{g_1,g_2}} + ({\mathbb
          T}|_{X^{g_1,g_2}})^{g_1,g_2}.
\end{equation}
Here ${\mathbb T}$ is the class in $K_0(G,X)$ corresponding to the
tangent bundle of the stack $\ix = [X/G]$.

\subsection{Proof of Theorem \ref{thm.orbifoldproducts}}

\begin{theorem} \label{thm.obstructionpullback}
Let $\iy \to \ix$ be a strong regular embedding and let
${\mathscr R_X}$ be the class of the obstruction bundle for
the orbifold product in $K_0(I^2\ix)$. Then 
${\mathscr R_Y} = I^2\iota^* {\mathscr R_X}$ where
${\mathscr R_Y}$ is the obstruction bundle for the orbifold
product on $\iy$ and $I^2\iota \colon I^2\iy \to I^2 \ix$ is the inclusion.
\end{theorem}

A necessary ingredient in the Proof of Theorem \ref{thm.obstructionpullback}
is the following property of strong regular embeddings.
\begin{proposition} \label{prop.normalpullsback}
Let $\iota \colon \iy \to \ix$ be a strong regular embedding 
of smooth Deligne-Mumford stacks
and let $I\iy$ and $I\ix$  be their respective inertia stacks.
Then the normal bundle of $I\iy$ in $I\ix$ is the pullback 
of the normal bundle of $\iy$ in $\ix$.
Likewise 
the normal bundle to $I^2\iy$ in $I^2\ix$ is the pullback
of the normal bundle of $\iy$ in $\ix$.
\end{proposition}
\begin{example}
The hypothesis that $\iy$ is smooth is crucial. Let
$\ix = [\A^2/\mu_2]$ where $\mu_2$ acts by $-1\cdot(a,b) = (-a,-b)$.
Let $V= V(xy)$  and let $\iv = [V/\mu_2]$. Since $V$ is defined by a
$\mu_2$-invariant function, the embedding $\iv \to \ix$ is a strong regular
embedding of pure codimension one. The inertia $I\ix$ is the quotient
disjoint
sum $I\ix \coprod B\mu_2$ and the inertia $I\iy$ is the disjoint sum
of $I\iy \coprod B\mu_2$. Thus the embedding of $I\iy$ in $I\ix$ does not
have pure codimension one so the normal bundle if $I\iy$ in $I\ix$ cannot be the pullback of the normal bundle of $\iy$ in $\ix$.
\end{example}

\begin{proof}[Proof of Proposition \ref{prop.normalpullsback}]
Since $\iy \to \ix$ is a closed embedding the diagram
$$\xymatrix{I\iy \ar[d]\ar[r]^{I\iota} & I\ix \ar[d]\\
\iy \ar[r]^{\iota} & \ix}$$
is cartesian. Moreover, since $\iy$ and $\ix$ are smooth,
$I \iy$ and $I \ix$ are also smooth, so $I\iota$ is also a regular embedding.
To check that the normal bundle of $I\iy$ in $I\ix$ is the pullback
of the normal bundle of $\iy$ and $\ix$ it suffices to prove
that the two embeddings have the same codimension.

This can be checked after base change by strongly \'etale morphism.
Thus we can assume that $\ix = U/G$ where $U$ is
affine and $G$ is a finite group and $\iy = [V/G]$ where $V$ is a
closed subscheme cut out by a $G$-fixed regular sequence.  Since $\ix
= [U/G]$ the inertia stack $I \ix$ is the quotient stack $[I_GU/G]$
where $I_GX$ is the inertia group for the action of $G$. Since $G$ is
finite $I_GU = \coprod_{g \in G}U^g$ where $U^g$ is the subscheme
fixed by $g$. Likewise, $I\iy = [I_G V/G]$ and $I_GV = \coprod_{g \in
  G} V^g$.

Thus it suffices to check that $V \hookrightarrow U$ is a
$G$-equivariant embedding of affine schemes of pure codimension $d$
and $V$ is defined by a $G$-fixed regular sequence. Then for any $g \in
G$ the codimension of $V^g$ in $U^g$ is also $d$.  Let $y \in V^g$ be
a point and let $G_y$ be the stabilizer of $y$ which necessarily
contains $g$.  The argument used in the proof of Theorem \ref{thm.strongembedd} shows
that the complete local ring $\widehat{\mathcal O}_{y,U}$ is
isomorphic as a $G_y$-module to $\widehat{\mathcal O}_{y,V}[[T_1,
\ldots ,T_d]]$ where $G_y$ acts trivially on the $T_i$. Hence
$$\widehat{\mathcal O}_{y,U^g} = \widehat{\mathcal O}_{y,U}^{\langle g \rangle} =
{\mathcal O}_{y,V}^{\langle g \rangle}[[T_1, \ldots , T_d]]$$
so the codimension is preserved.

A similar argument shows that the normal bundle of $I^2 \iy$
in $I^2 \ix$ is also the pullback of the normal bundle
to $I\iy$ in $I\ix$.
\end{proof}

\begin{proof}[Proof of Theorem \ref{thm.obstructionpullback}]
   If $\iy \hookrightarrow \ix$ is a strong regular embedding then by
  Proposition \ref{prop.normalpullsback} we have the following
  identity in $K_0(Z_G(g_1,g_2),(Y^{g_1,g_2}))$: 
$${\mathbb
    T\ix}|_{Y^{g_1,g_2}} = {\mathbb T\iy}|_{Y^{g_1,g_2}} +
  N|_{Y^{g_1,g_2}}$$ 
where $N$ is the normal bundle of $\iy$ in $\ix$.
Moreover, since $\iota$ is a strong regular embedding the action of the fibers
  of $N|_{Y^{g_1,g_2}}$ are trivial modules for the action of the
  group generated by $g_1,g_2$. Hence $L(g_1)(N) = L(g_2)(N) =
  L((g_1g_2)^{-1})(N) = 0$. Also, since $g_1, g_2$ act trivially on
  $N$, $N= N^{g_1,g_2}$.  Substituting into formula
  \eqref{eq.logtrace} we see that $LT(\mathbb T_\ix)|_{Y^{g_1,g_2}} =
  LT(\mathbb T_\iy)$ which proves the proposition.
\end{proof}


\begin{proof}[Proof of Theorem \ref{thm.orbifoldproducts}]
To prove the theorem we must show the following identity holds
for any $\alpha, \beta \in \CH^*(I\ix)$.
\begin{equation} \label{eq.crucial} \mu_* \left(e_1^*(I\iota)^*(\alpha) \cdot
e_2^*(I\iota)^*(\beta) \cdot \eu({\mathscr R}_\iy) \right)
= (I^2\iota)^*\left(\mu_*( e_1^*\alpha \cdot e_2^*\beta \cdot \eu({\mathscr R}_\ix)\right)
\end{equation}
By Theorem \ref{thm.obstructionpullback} ${\mathscr R_\iy} =
(I^2\iota)^*{\mathscr R_\ix}$.
By Proposition \ref{prop.normalpullsback} the normal bundle
of $I^2\iy$ in $I^2\ix$ is the pullback of the normal bundle
to $I\iy$ in $I\ix$. It follows \cite[Theorem 6.2(b), (c)]{Ful:84} that
$e_i^* \circ (I\iota)^* = (I^2\iota) \circ e_i^*$ as morphisms
$\CH^*(I\ix) \to \CH^*(I^2\iy)$. Applying \cite[Theorem 6.2(a), (b)]{Ful:84}
also implies that
$\mu_* \circ (I^2\iota)^* = (I\iota)^* \circ \mu_*$
as morphism $\CH^*(I^2\ix) \to \CH^*(I\iy)$.

Substituting these identities into Equation \eqref{eq.crucial} yields
the theorem.
\end{proof}

\subsection{Examples}
Let $\ix = [\A^3/\mu_3]$ where the generator $\omega$ of $\mu_3$ acts
by $\omega(a,b,c) = (a, \omega b, \omega^2 c)$.  Since $\A^3$ is a
representation of $\mu_3$, $\CH^*(\ix) = \CH^*(B\mu_3) = \ZZ[t]/3t$.
The inertia $I \ix$ has three components indexed by $\mu_3$.  The
identity section $I_1$ is isomorphic $\ix$ and the components
$I_\omega$ and $I_{\omega^2}$ which are both isomorphic to $\A^1
\times B\mu_3$ which is identified with the quotient
$[\{(a,0,0)\}/B\mu_3]$.  The Chow groups $\CH^*(I\ix)$ have a natural
$\CH^*(B\mu_3)$-module structure which is preserved by the orbifold
product so we write $\CH^*(I \ix) =\oplus_{m \in \mu_3} \ZZ[t]/(3t)$.

Since $\mu_3$ is abelian there are 9 double conjugacy classes and
$I^2\ix$ has 9 components indexed by $\mu_3 \times \mu_3$.  For each
pair $(m_1, m_2) \in \mu_3 \times \mu_3$ let ${\mathscr R}(m_1, m_2)$
be the restriction of the obstruction bundle to the component
$I^2_{m_1,m_2} = [(\A^3)^{m_1,m_2}/\mu_3]$.  Let $\xi$ be the defining
representation of $\mu_3$.  Then the tangent bundle ${\mathbb T}$ of
$\ix$ corresponds to the representation $1 + \chi + \chi^2$ where $1$
denotes the trivial representation. Using the formula of \cite{EJK:10}
we obtain ${\mathscr R}(\omega, \omega) = \chi^2$ ${\mathscr
  R}(\omega^2, \omega^2) = \chi$ and ${\mathscr R}(m_1,m_2) = 0$ for
all other pairs $(m_1, m_2) \ \in \mu_3 \times \mu_3$.

If we denote by $l_m$ with $m \in \mu_3$ the generator of the
$\ZZ[t]/(3t)$-module $\CH^*(I_m)$, then we obtain the following
identities for the orbifold product.
\begin{eqnarray*}
l_\omega \star l_\omega  = 2t l_\omega\\
l_\omega \star l_{\omega^2} = 2t^2 l_1\\
l_{\omega^2} \star l_{\omega^2} = t l_\omega
\end{eqnarray*}
and $l_1$ acts as the identity.
Hence $$\CH^*_{orb}(I\ix) = \ZZ[t, l_\omega, l_{\omega^2}]/(3t, 
l^2_{\omega} - 2t l_\omega,
l_\omega l_{\omega^2} - 2t^2, l^2_{\omega^2} - t l_{\omega^2}).$$

We will now consider various substacks of $\ix$ and compare orbifold
Chow rings. To start let $\iy = [\A^2/\mu_3]$ where $\omega \cdot
(b,c) =(\omega b, \omega c)$ The map $\A^2 \to \A^3$, $(b,c) \mapsto
(0,b,c)$ induces a strong regular embedding of stacks $\iy \to \ix$
since the ideal of $\A^2$ in $\A^3$ is generated by the $\mu_3$-fixed
function $x$.  Again, $\CH^*(I\iy)$ is isomorphic to the $\oplus_{m \in
  \mu_3} \ZZ[t]/(3t)$ and the pullback induced by the inclusion $I\iy
  \to I\ix$ maps $l_\omega$ to $l_\omega$ and $l_{\omega^2}$ to
  $l_{\omega^2}$ and direct calculation shows that the pullback
  induces an isomorphism $\CH^*_{orb}(\iy)\to
\CH^*_{orb}(\ix)$.

Now let $\iy$ be the substack $[\A^2/\mu_3]$ where $\omega \cdot (a,b) = 
(a, \omega b)$. The map $\A^2 \to \A^3$, $(a,b) \mapsto (a,b,0)$
induces an embedding $\iy \to \ix$ but since the defining equation
of $\A^2$ in $\A^3$ is not $\mu_3$-invariant this is not a strong regular embedding. 
Indeed a direct calculation shows that 
${\mathscr R}(\omega, \omega) = 0$ so
that ${\mathscr R}_\iy$ is not the pullback of ${\mathscr R}_\ix$.
Again, $\CH^*(I\iy)$ is isomorphic to $\oplus_{m \in \mu_3} \ZZ[t]/(3t)$
and the pullback induced by the inclusion $\iy \to \ix$ maps $l_m$ to $l_m$.
However, because ${\mathscr R}_{\omega, \omega} = 0$
the orbifold Chow ring has the following presentation
$$\CH^*_{orb}(I\iy) = \ZZ[t,l_1, l_2]/(3t, l_{\omega^2} - l_\omega, l_\omega l_{\omega^2} - 2 t^2, l_{\omega^2} l_{\omega^2} - t l_{\omega^2})$$
and we see that the map that sends the generator $l_m$ of $\CH^*(I\ix_m)$ to the
generator $l_m$ of $\CH^*(I\iy_m)$ is not a homomorphism of orbifold Chow
rings.

Finally let $\iy = [Z/\mu_3]$ where $Z = V(yz-1) \subset \A^3$.  Since
the equation $yz -1$ is $\mu_3$-fixed the map $\iy \to \ix$ is a
strong regular embedding.  Since $\mu_3$ acts freely on $Z$ on the
quotient stack $\iy$ is represented by the scheme $\A^1 \times \G_m$.
Thus $I\iy = \iy$ and so $(I\iy)_\omega$ and $(I \iy)_{\omega^2}$ are
both empty.  Thus $\CH^*(I\iy) = \CH^*(\A^1 \times \G_m) = \ZZ$ and the
orbifold product is trivial.  The pullback induced by the inclusion
$\iy \to \ix$ maps $l_\omega$ and $l_{\omega^2}$ to 0 and the map on
orbifold Chow rings is the homomorphism
$$\ZZ[t,l_\omega, l_\omega^2]/(3t, l^2_{\omega} - 2t l_\omega, l_\omega l_{\omega^2} - 2t^2, l_{\omega^2}^2 - t l_{\omega^2}) \to \ZZ$$
given by setting $t, l_\omega, l_{\omega^2}$ to be equal to 0. 

\section{Application: The generalized Euler sequence and cotangent bundles stacks} \label{sec.cotangent}
Let $T$ be a torus acting properly on a smooth variety
$X$ defined over a field $k$. Let $\ix = [X/T]$ be the corresponding
quotient stack. 


In this section we consider two natural quotient stacks that arise
from this data, namely $T^*\ix$ and $[T^*X/T]$. The first stack is
intrinsic to $\ix$ while the latter depends on the presentation of
$\ix$ as a quotient stack. There is an exact sequence
of vector bundles on $\ix$
\begin{equation}
\label{eq.eulerseq} 0 \to T^*\ix \to [T^*X/T] \to [X \times \Lie(T)^*/T] \to 0.
\end{equation}
Since $T$ is a torus $\Lie(T)$ is a trivial $T$-module
and the last bundle in the sequence is the trivial bundle $[X/T] \times Lie(T)$.
We call \eqref{eq.eulerseq}  the {\em generalized Euler sequence} because, when
$X = \A^{n+1} \smallsetminus \{0\}$ and $T = \G_m$ acts with weights
all equal to one then $\ix = [X/T] = \Pro^n$ and \eqref{eq.eulerseq}
is the usual Euler sequence for the cotangent bundle of $\Pro^n$.

Our goal is to understand the relationship between the stacks $T^*\ix$ and
$[T^*X/T]$. 
Since both stacks are vector bundles over $\ix$ the
inclusion $i \colon T^*\ix \to [T^*X/T]$ induces  pullback
isomorphism
$i^*\colon \CH^*([T^*X/T] \to \CH^*(\ix)$.

However, stronger results hold.
Let ${\mathbf M}$ be the coarse space of $T^*\ix$ and let
${\mathbf N}$ be the coarse space of $[T^*X/T]$. 
By Sumihiro's theorem \cite{Sum:74}, the action of a torus on a normal variety is locally linearizable. This implies that the coarse spaces ${\mathbf M}$, 
${\mathbf N}$
and ${\mathbf X}$ are all $k$-varieties.

There is
commutative triangle of morphisms of schemes.
$$\xymatrix{
{\mathbf M} \ar[rr]^i \ar[rd] & & {\mathbf N} \ar[ld] \\
& \bx &
}
$$
Note that in general ${\mathbf M}$ and ${\mathbf N}$
are not vector bundles over $\bx$.
\begin{proposition}  The inclusion $\iota
  \colon T^*\ix \to [T^*X/T]$ is a strong regular embedding, so
the inclusion ${\mathbf M} \to {\mathbf N}$ is a regular embedding which is \'etale locally isomorphic to a section of a smooth morphism.
\end{proposition}
\begin{proof}
  Observe that the sequence $ T^*\ix \to [T^*X/T] \to [X \times
  \Lie(T)^*/T]$ is exact and $\Lie(T)^*$ is a trivial
  $T$-module (because $T$ is diagonalizable). Hence the normal bundle to
  $\iota$ is the trivial bundle $T^*\ix \times \Lie(T)^*$
so $\iota$ is a strong regular embedding.
\end{proof}

\begin{corollary}
If $\tilde{N}$ is a canonical (functorial) resolution of singularities of $N$
then $\tilde{M} := M \times_N \tilde{N}$ is a resolution of singularities as
well.
\end{corollary}

The main result  of this section is the following result about
the Chow groups of ${\mathbf M}$ and ${\mathbf N}$.
\begin{theorem} \label{thm.cotangentpackage}
If ${\mathbf N}$ is a scheme then the pullback $\iota^* \colon \CH_*(\bn) \to \CH_{* - \dim T}(\bm)$ is an
isomorphism of integral Chow groups. 
\end{theorem}

Theorem \ref{thm.cotangentpackage} follows from a more general 
result about strong regular embeddings of vector bundles
on not-necessarily smooth quotient stacks.

Let $\ix = [X/T]$ be a reduced, tame quotient stack with finite stabilizer.
Let
$$\xymatrix{0 \ar[r] & \ar[dr] \ar[r] \iv' & \ar[d] \ar[r] \iv & \ar[dl] \iv'' \ar[r] & 0\\
& & \ix & & & 
}$$
be a short exact sequence of vector bundles on $\ix$ such
that there is a stratification of $\ix$ on which $\iv''$ is trivial
so the inclusion $\iv' \to \iv$ is a strong regular embedding by Theorem
\ref{thm.strongembedd}.
Denote by $\bV$ and $\bV'$ the coarse moduli spaces of $\iv$ and $\iv'$
respectively and let $i \colon \bV' \to \bV$ be the inclusion.
\begin{theorem} \label{thm.quotientvbs}
If $\bV$ is a scheme
the pullback on (higher) Chow groups 
$i^* \colon \CH_*(\bV,i) \to \CH_{*-d}(\bV',i)$ is an isomorphism with integer coefficients. 
\end{theorem}

\begin{remark}
The assumption that ${\bV}$ is a scheme is required for the proof
because we use the localization theorem for higher Chow groups. If
we knew that this sequence was also valid for algebraic spaces then
the theorem would go through without this hypothesis. 

\end{remark}

We first prove a special case of Theorem \ref{thm.quotientvbs}
\begin{lemma} \label{lem.gerbecase}
The conclusion of Theorem \ref{thm.quotientvbs} holds
for classifying stacks $BH$ where $H$ is finite, linearly reductive group.
\end{lemma}
\begin{proof} A vector
  bundle on $BH$ is a stack $\iv = [V/H]$ where $V$ is a linear
  representation of $H$.  The coarse space of $\iv$ is the quotient
  scheme $V/H = \Spec {\mathcal O}(V)^H$.

By hypothesis we are given a short exact sequence 
$$0 \to V' \to V \to V'' \to 0$$
of $H$-modules
such that $V''$ is trivial.

Since $H$ is linearly reductive any short exact sequence of $H$-modules
splits, so $V = V' \oplus V''$ and the inclusion $V' \to V=V' \oplus V''$
is the 0-section of the projection $V' \oplus V''\to V'$.
Moreover, since $H$ acts trivially on $V''$,
$V/H = V'/H \times V''$ and the inclusions of quotient
$V'/H \to V/H$ is the inclusion of the 0-section of the trivial vector bundle
$V'/H \times V \to V'/H$. Since pullback along the 0-section of a vector bundle
induces isomorphisms of (higher) Chow groups.
\end{proof}

\begin{lemma} \label{lem.thomason}
  If $\ix = [X/T]$ is a reduced Deligne-Mumford quotient stack with $T$ a diagonalizable
group
  then there is a dense open substack $\iu \subset \ix$ such that $\iu$ is
  isomorphic to $BH \times W$ where $H$ is a finite diagonalizable
  group and $W$ is affine.
\end{lemma}
\begin{proof}[Proof of Lemma]
Embedding $T$ into a torus $T'$ the stack $\ix$ can presented as the quotient
$[(X \times_T T')/T']$. Since $T$ and $T'$ are assumed to be smooth,
$X \times_T T'$ is also reduced. Replacing $T$ with $T'$ we may assume that it
is a torus. By \cite[Lemma 4.3]{Tho:86d} there is a dense $T$-invariant
open subspace of $X$ which is a separated scheme. Replacing
$\ix$ by this open subset we may assume that $X$ is a reduced, separated scheme.

By \cite[Proposition 4.10]{Tho:86d} there is an affine open subspace
$U \subset X$ and a diagonalizable subgroup $T' \subset T$ with
quotient torus $T''$ such that $T$ acts via the homomorphism $T\to
T''$ and $T''$ acts freely on $U$. If we let $W = U/T$, then $U$ is
$T$-equivariantly isomorphic to $T/T' \times W$.  Since $[U/T]$ is
Deligne-Mumford it follows that $H=T'$ is a finite group and so we see
that $[U/T]$ is equivalent to $BH \times W$.
\end{proof}

Let $\iv$ be a vector bundle on $\ix$ and let $\iu$ be an open
substack. Let $\iv_{\iu}$ denote the restriction of $\iv$ to a vector
bundle on $\iu$ and let $\bV_{\iu}$ be its coarse moduli space.

\begin{lemma} \label{lem.genericcase}
Let $\iv' \to \iv \to \iv''$ be an exact sequence of vector bundles
on $\ix = [X/T]$ satisfying the hypothesis of the theorem.
Then there is a dense open substack $\iu \subset \ix$
such that the pullbacks $i_\iu^* \colon \CH_*(\bV_\iu) \to \CH_*(\bV'_\iu)$
is an isomorphism.
(Here $i_\iu \colon \bV'_\iu \to \bV_\iu$ is the inclusion of coarse spaces.)
\end{lemma}
\begin{proof}
By Lemma \ref{lem.thomason} there is a dense open substack $\iu \subset \ix$
isomorphic to $BH \times W$ where $H$ is a finite diagonalizable group
and $W$ is affine. The restriction of a vector bundle $\iv$
to $U$ is a sum of vector bundles of the form $[V/H] \otimes {\mathcal E}$
where $V$ is a representation of $H$ and ${\mathcal E}$ is a vector bundle
on $W$. Therefore, given an exact sequence $\iv' \to \iv \to \iv''$
we can shrink $W$ and hence $\iu$ so that each of these bundles
restricted to $\iu$ is of the form $V \otimes {\mathcal O}_W$. 
We can then repeat the argument of Lemma \ref{lem.gerbecase}.
\end{proof}

\begin{lemma} \label{lemma.quotientlocalization} Let $\ix$ be a
  Deligne-Mumford stack with finite stabilizer and let $\iu$ be an open set with complement
  $\iz$ (with the reduced induced substack structure). Let $\bx, \bu,
  \bz$ be the coarse spaces of $\bx, \bu$ and $\bz$
  respectively.  Then the inclusion $\iz \to \ix$ (resp. $\iu \to
  \ix$) induces a closed (resp. open) immersion $\bz \to \bx$
  (resp. $\bu \to \bx$) and $\bu = \bx \smallsetminus \bz$.
\end{lemma}

 \begin{proof}[Proof of Lemma \ref{lemma.quotientlocalization}] Again
     using the structure theorem for coarse moduli spaces we can
     reduce to the case that $\ix = [\Spec A/G]$ for some finite group
     $G$. Then $\iu = [W/G]$ where $W$ is a $G$-invariant open
     set. Let $Y = \Spec A \smallsetminus W$. The coarse moduli space
     of $\iu$ is $\pi(W)$ where $\pi \colon \Spec A \to \Spec A^G$ is
     the quotient map. Likewise the coarse moduli space of $\iz$ is
     $\pi(Y)$.
     
     Since the morphism $\pi$ is a geometric quotient its geometric
     fibers are $G$-orbits of geometric points. It follows that
$\pi(W)$ is open in $X = \Spec A^G$ and equal to 
$\pi(\Spec A) \smallsetminus \pi(Z)$.
\end{proof}

\begin{proof}[Proof of Theorem \ref{thm.quotientvbs}]
Observe that if $\iv$ is a vector bundle
on $\ix$ and  $\iu \subset \ix$ is an open substack then
the complement of $\iv|_{\iu}$ (with its reduced induced stack structure)
is $\iv_{\iz}$ where $\iz$ is the complement of $\iu$ with its reduced induced stack structure. Hence by Lemma \ref{lemma.quotientlocalization}
the complement of the coarse space $\bV_\iu$ is the coarse space
$\bV_\iz$. 

Now by Lemma \ref{lem.genericcase} there is an open substack $\iu
\subset \iv$ such that the inclusion $\bV'_{\iu} \to \bV_{\iu}$
induces a pullback isomorphism of (higher) Chow groups
$\CH^d(\bV'_\iu,i) \to \CH^d(\bV_{\iu},i)$ for every $d, i$.  By
Noetherian induction we may assume that the inclusion $\bV'_{\iz} \to
\bV_{\iz}$ also induces a pullback isomorphism of (higher) Chow
groups.  The theorem follows by applying the localization exact
sequence for higher Chow groups \eqref{eq.localization}.
\end{proof}

\begin{proof} [Proof of Corollary \ref{cor.cotangent}]
By Theorem \ref{thm.orbifoldproducts} we know that the
pullback $(I\iota)^*  \colon \CH^*(I(T\ix) \to I([T^*X/G])$
commutes with the orbifold product. Thus it suffices
to show that $(I\iota)^*$ is an isomorphism of abelian groups.
This follows because both $T^*\ix$ and $[T^*X/G]$ are vector bundles
over $\ix = [X/G]$ so $IT^*\ix$ and $I[T^*X/G]$ are both vector bundles
over $I\ix$. Hence the pullback $I\iota^*$ is an isomorphism.
\end{proof}

\begin{remark} The methods used to prove Theorem \ref{thm.cotangentpackage}
and Theorem \ref{thm.quotientvbs} yield analogous isomorphisms
for the $K$-theory of coherent sheaves. Since $K$-theory is naturally
defined for schemes over an arbitrary base we do not need to assume that the base scheme
is a field. Also, there is a localization long exact sequence for the 
higher $K$-theory of algebraic spaces so we can obtain the $K$-theory
result in more generality. In particular, we may assume that $X$ is a smooth
algebraic space defined over a Noetherian scheme $S$
and that $T$ is diagonalizable - i.e.,  isomorphic to a closed subgroup
of a torus.
\end{remark}

\section{Application: Hypertoric Stacks and Lawrence Toric Stacks}
\label{sec.hyperkahler}
A natural algebraic way to
construct hypertoric varieties is via geometric invariant theory (GIT)
as follows:

Fix integers $d,n$ with $n \geq d$ and let $V$ be an $n$-dimensional
representation of the rank $d$ torus $T$. Choosing a diagonalizing
basis for the action of $T_d$ let $A = (a_{ij}) \in \ZZ^{d \times n}$
be the matrix of weights for the action of $d$. We denote by $a_k$ the $k$-th
column vector of $A$. We assume that $A$ has
maximal rank over $\QQ$ which is equivalent to assuming that the
generic stabilizer is finite. 

There is a natural action of $T$ on $V \times V^*$ and with an
appropriate choice of basis the weight matrix for this action is
$A^\pm = (a_1 \ldots a_n -a_1 \ldots -a_n) \in \ZZ^{d \times 2n}$.

If $\theta \in \ZZ^d$ is a character of $T$ then we can consider the
$\theta$-stable and semi-stable loci in $V \times V^*$. For generic
choice of $\theta$, $(V \times V^*)^{s} = (V \times V^*)^{ss}$. Note
that the coordinate ring of $V \times V^*$ 
contains invariant elements so the GIT quotient $\bx(A^\pm, \theta)
:=(V \times V^*)_\theta//T_d$ is not projective. However, it is
semi-projective meaning that it is semi-projective 
over $\Spec \Sym((V \times V^*))^{T}$ and the quotient is a
$2n -d$-dimensional toric variety called a {\em Lawrence toric
  variety}.  The quotient stack $\ix(A^{\pm},\theta) := [(V \times
V^*)^{s}/T]$ is called a {\em Lawrence toric stack} and its coarse
moduli space is $\bx(A^{\pm}, \theta)$.

The representation $V \times V^*$ has a natural $T$-invariant 
algebraic symplectic
pairing $\mu \colon (V \times V^*) \to \Lie(T_d)^* = \CC^d$. If we
choose coordinates $(x_1, \ldots ,x_n)$ on $V$ and dual coordinates
$(y_1, \ldots , y_n)$ on $V^{*}$ so that $T_d$ acts on $x_i$ with
weight $a_i$ and on $y_i$ with weight $-a_i$ then $\mu$ is given by
the formula
$$(x_1, \ldots , x_n, y_1, \ldots y_n) \mapsto (\mu_1, \ldots , \mu_n)$$
where $\mu_i = \sum_{j} a_{ij} x_j y_j$.

The quotient $\by(A,\theta) = \mu^{-1}(0)//_\theta T$ is the associated
{\em hypertoric variety}. 
Following Jiang and Tseng \cite{JiTs:08} we refer to the quotient stack
$$\iy(A,\theta) :=
[\left(\mu^{-1}(0) \cap (V \times V^*)^s\right)/T]$$
as a hypertoric stack. The variety $\by(A,\theta)$ is the 
coarse space of $\iy(A,\theta)$.

\begin{example}
  If $A = (a_0, \ldots a_n)$
is a $1 \times (n+1)$ matrix with all $a_i$ positive,  denote by
  $\Pro(a_0, \ldots , a_n)$ the quotient stack
  $[\A^{n+1}\smallsetminus \{0\}/\G_m]$ where $\G_m$ acts with weights
$(a_0, \ldots , a_n)$. 
If $\theta$ is positive then 
$\iy(A,\theta) =T^*\Pro(a_0,\ldots , a_n)$ and the Lawrence toric stack
$\ix(A^\pm, \theta)$ equals $[T^*(\A^{n+1}\smallsetminus \{0\}))/T]$.
\end{example}

\begin{theorem}\label{thm.localstructure}

The stack $\ix(A^\pm,\theta)$ has a Zariski open cover by open
sets  $\iu$ 
each isomorphic to $(\iy(A,\theta) \cap \iu) \times \A^d$ such that 
under this isomorphism $\iota$ corresponds to a section of the projection
$\iu \to (\ix(A,\theta) \cap \iu)$.

In particular the inclusion $\iota \colon \iy(A,\theta)
\hookrightarrow \ix(A^{\pm}, \theta)$ is a strong regular embedding
of smooth Deligne-Mumford stacks.
\end{theorem}

Applying Corollary \ref{cor.resolution} we obtain as a corollary the following
result about resolutions of singularities.
\begin{corollary}
  Let $\tilde{\bx} \to \bx(A^{\pm},\theta)$ be a canonical resolution
  of singularities of the Lawrence toric variety $\bx(A^{\pm},\theta)$
  and set $\tilde{\bx} := \by(A,\theta) \times_{\bx(\A^{\pm,\theta)}}
  \tilde{\bx}$. Then $\tilde{\by} \to \by(A,\theta)$ is a
  resolution of singularities of the hypertoric variety $\by(A,\theta)$.
\end{corollary}

As in the case of cotangent bundles we also have a result
about integral Chow groups and 
Chow rings as well.
\begin{theorem} \label{thm.chowhypertoric}
The pullback $\iota^* \colon \CH^*(\ix(A^\pm, \theta)) \to \CH^*(\iy(A,\theta))$
is an isomorphism of integral Chow rings of Deligne-Mumford stacks.

Also, if $\by(A,\theta) \stackrel{i} \hookrightarrow \bx(A^{\pm}, \theta)$ 
is the inclusion of a hypertoric variety into the corresponding Lawrence toric variety then
the pullback $i^* \colon \CH_k(\bx(A^{\pm},\theta)) \to \CH_{k-d}(\by(A,\theta))$ is
an isomorphism of integral Chow rings.


The analogous statements also hold for the $K$-theory of coherent sheaves.
\end{theorem}

\begin{proof}[Proof of Theorem \ref{thm.localstructure}]
By \cite[Corollary 4.4]{HaSt:02} 
a point is $\theta$-stable if it lies in the complement of the 
vanishing of the irrelevant ideal
$$B_\theta = \langle \prod_{\sigma(C,\theta)}: C \text{ any column basis of } 
A\rangle$$ 
where $\sigma(C,\theta)$ is defined as follows: If $C = \{a_{i1},
\ldots , a_{id}\}$ then there are unique non-zero\footnote{The assumption that $\theta$ is generic ensures that all of the $\lambda_i$ are non-zero. The condition that the $\lambda_i$ are non-zero for each column basis of $A$ is equivalent to the condition that every $\theta$-semi-stable point is stable.} 
rational numbers
$\lambda_1, \ldots , \lambda_d$ such that
$$\lambda_1 a_{i1} + \lambda_2 a_{i2} + \ldots + \lambda_d a_{id} = \theta$$
and 
$$\sigma(C, \theta) = \{x_{i_j} \colon \lambda_j > 0\} \bigcup \{y_{i_l} \colon \lambda_l < 0\}.$$
Each set of indices in $\sigma(C,\theta)$ corresponds to a maximal 
cone in the fan of the toric stack $\ix(A,^\pm, \theta)$.
Let $U_{\sigma(C,\theta)} \subset V \times V^*$ be the principal open set
corresponding to the monomial $x_{\sigma(C,\theta)} = \prod_{\sigma(C,\theta)} x_{i_j}y_{i_l}$ and let $\iu_{\sigma(C,\theta)} = [U_{\sigma(C,\theta)}/T]$.
By construction the $U_{\sigma(C,\theta)}$ form an open cover of
$(V \times V^*)^s$. Theorem \ref{thm.localstructure} then follows from the following Proposition.
\end{proof}
\begin{proposition} \label{prop.crucial}
Let $\sigma = \sigma(C,\theta)$ for some fixed subset $C$ of the columns 
of $A$. If $W \subset U_\sigma$ is a $\G_m^{2n}$-invariant irreducible 
closed subset transverse to $\mu^{-1}(0)$
then $W$ is $T$-equivariantly isomorphic to $(\mu^{-1}(0) \cap W) 
\times {\mathbb A}^d$ where $T$ acts trivially on $\A^d$. Under
this isomorphism the inclusion $W\cap \mu^{-1}(0) \hookrightarrow W$
corresponds to a section of the projection $(\mu^{-1}(0) \cap V) \times \A^d \to \mu^{-1}(0) \cap V$.
\end{proposition}
\begin{proof}[Proof of Proposition \ref{prop.crucial}]
After reordering the coordinates and we may assume that
$x_{\sigma} = (x_1 x_2 \ldots x_k)(y_{k+1} y_{k+2} \ldots y_d)$
for some $k$ with $0 \leq k \leq d$. Since $A$ has rank $d$
we may also perform row operations over $\ZZ$ to ensure that $a_{ii} \neq 0$
for $i \leq d$.
Then on $U_\sigma$ the hypertoric equations $\{\sum a_{ij}x_jy_j\}_{i=1}^d$
can be rewritten as 
$x_i + {1\over{y_i}}\sum_{j\neq
    i}(a_{ij}/a_{ii}) x_jy_j$ if $i > k$
and
$y_i + {1\over{x_i}}\sum_{j\neq
    i}(a_{ij}/a_{ii}) x_jy_j$ if
$l \leq k$. 
Then we can define an isomorphism
$(\mu^{-1}(0) \cap U_{\sigma})\times \A^d \to U_{\sigma}$ 
where $T$ acts trivially on the second factor by the formula
$$\begin{array}{ll} 
\left((x_1,\ldots ,x_n, y_1, \ldots , y_n), (z_1, \ldots , z_d)\right)  \mapsto   \\
 (x_1, \ldots , x_k, {z_{k+1}\over{y_{k+1}}}, \ldots, {z_d\over{y_d}}, x_{d+1},
\ldots x_n, {z_1\over{x_1}}, \ldots , {z_k\over{x_k}}, y_{k+1}, \ldots , y_n)
\end{array}$$

Now let $W \subset U_{\sigma(C,\theta)}$ be a non-empty irreducible
$\G_m^{2n}$-invariant closed subset transverse to $\mu^{-1}(0)$.
Then $W = U_{\sigma} \cap V(x_{i_1},\ldots , x_{i_l}, y_{j_1}, \ldots , y_{j_m})$ for some sets indices $\{i_1, \ldots , i_l\}, \{j_1, \ldots , j_m\}$,
with $\{i_1 \ldots , i_l\} \cap \{1, \ldots , k\} = \emptyset$ and
$\{j_1, \ldots j_m\} \cap \{ k+1, \ldots , d\} =\emptyset$.

If $\{i_1, \ldots, i_l\} \cap \{k+1, \ldots , d\} = \emptyset$ and 
$\{j_1, \ldots , j_m\} \cap
\{1, \ldots , d\} = \emptyset$ as well, then the argument used above 
shows that there is a $T$-equivariant isomorphism
$W \to (W\cap \mu^{-1}(0)) \times \A^d$.

If this is not the case we can assume, without loss of generality,
that $j_1 \in \{1, \ldots , d\}$ and after reordering coordinates we
have that $j_1 = 1$.  Let $V_1 \subset V$ be the submodule where $x_1
=0$ and let $A'$ be the matrix obtained by deleting the first column
of $A$.  Then $W$ is a closed subset of $\G_m \times (V_1 \times
V_1^*)$ and the hypertoric equations $\sum a_{ij}x_jy_j$ restrict on
$V_1 \times V_1^*$ to the equations for $\iy(A',\theta)$ in
$\ix(A'^\pm, \theta)$. The transversality assumption ensures that the
hypertoric equations do not degenerate on $W$. Therefore, by induction
on $\dim V$ we can conclude that the conclusion of the Lemma holds for
the embedding $\mu^{-1}(0) \cap W \hookrightarrow W$.
\end{proof}

\begin{proof}[Proof of Theorem \ref{thm.chowhypertoric}]
The proof is similar to the proof of \cite[Theorem 1.1]{HaSt:02}.
Since ${\mathbf X}(A^\pm, \theta)$ is a semi-projective toric variety
it is the toric variety of a fan $\Sigma_\theta$ in a lattice $N$ which has full dimensional support. 
A vector $v \in N$ determines a $1$-parameter subgroup $\lambda_v$
which acts on the toric variety ${\mathbf X}(A^\pm, \theta)$. If $v$ is
in the support of the fan $\Sigma_\theta$ then the action on $\bx(A^\pm, \theta)$
is filterable. This means we can order the fixed components 
for the action of $\lambda_v$ so that Bialynicki-Birula decomposition
of $\bx(A^\pm, \theta)$ with respect to this action gives a filtration
$$ \emptyset = \bU_0 \subset \bU_1 \subset \ldots \subset \bU_r = \bx(A^\pm,\theta)$$
such that for every $i \geq 1$, the set $\bU_i \smallsetminus \bU_{i-1}$
is a closed subset of $\bU_\sigma$ where $\sigma$ is a maximal cone in the fan of
$X$.

Now if $\lambda_v$ is the one-parameter subgroup corresponding to the sum 
of the generators for the rays in the fan $\Sigma_\theta$ then
$\lambda_v$ is the image of diagonal one-parameter subgroup 
of the torus $\G_m^{2n}$ under the quotient map $\G_m^{2n} \to \G_m^{2n-d}$
\cite[Proof of Lemma 6.5]{HaSt:02}.
Since the hypertoric equations are homogeneous in the variables
$x_1,\ldots , x_n, y_1, \ldots , y_n$
it follows that the quotient $\by(A,\theta)$ is invariant under the action
of $\lambda_v$.

As observed by Hausel and Sturmfels, the fixed loci for the $\lambda_v$
actions on $\bx(A^\pm, \theta)$ and $\by(A,\theta)$ are equal
since these two semi-projective varieties have a common core. 

The Bialynicki-Birula decomposition of $\by(A,\theta)$
gives a filtration
$$ \emptyset = \bV_0 \subset \bV_1 \subset \ldots \subset \bV_r = \by(A,\theta)$$
such that $\bV_i \smallsetminus \bV_{i-1}$ is closed in $\bU_i
\smallsetminus \bU_{i-1}$. Moreover, since the hypertoric equations
have positive weight with respect to the action of $\lambda_v$ the
codimension of a component of the Bialynicki-Birula decomposition
of $\by(A,\theta)$ in the corresponding component of the Bialynicki-Birula
decomposition of $\bx(A^\pm,\theta)$ is constant.

By construction of ${\mathbf X}(A^\pm, \theta)$ as a geometric
invariant theory quotient, the inverse image of the affine toric open
subset $\bU_\sigma$ is an open set $U_{\sigma(C,\theta)}$ for some
column basis $C$ of $A$. Likewise any torus-invariant closed subset
of $\bW \subset \bU_{\sigma}$ appearing in the filtration 
corresponds to a $\G_m^{2n}$ invariant
closed subset $W \subset U_{\sigma(C,\theta)}$ which is transverse to $\mu^{-1}(0)$.

By Proposition \ref{prop.crucial} there is a $T$-equivariant
isomorphism $(\mu^{-1}(0) \cap W) \times \A^d \to W$ where $T$ acts
trivially on second factor.  Hence the pullbacks $\iota^*$ and $i^*$
induce isomorphisms of (higher) Chow groups $\CH_{j}(\iu_{i+1}
\smallsetminus \iu_{i},k) \to \CH_{j-d}(\iv_{i+1} \smallsetminus
\iv_i,k)$ and $i^* \colon \CH_j(\bu_{i+1} \smallsetminus \bu_i,k) \to 
\CH_{j-d}(\bv_{i+1} \smallsetminus \bv_i, k)$ for all for
all non-negative integers $j,k$ where 
$\iu_i= [U_i/T]$ and $\iv_i = [V_i/T]$.  



Using induction on the stratifications
$\bU_{\cdot}$ and $\bV_{\cdot}$ 
and the localization long exact sequence for higher
Chow groups we see that pullback of (higher) Chow groups
$i^* \colon \CH_*(\bX(A^\pm,\theta),k) \to \CH_*(\bY(A,\theta),k)$
is an isomorphism. 

The same argument also works for the (higher) $K$-theory
of coherent sheaves.
\end{proof}

\begin{remark} Although this result is a Chow group analogue of 
Hausel and Sturmfels's earlier result on
  cohomology \cite{HaSt:02}, the Chow and cohomology groups of a singular
  toric variety need not be equal. 
Also observe that the same methods can be used to show that there are
corresponding isomorphisms in algebraic $K$-theory.
\end{remark}
\begin{remark}
  The proof of Theorem \ref{thm.chowhypertoric} makes crucial use of
  the fact that a Lawrence toric variety is projective over 
  the affine toric variety $\Spec (V \times V^*)^T$, since this condition
implies that it has a stratification by unions of torus orbits. It would be
  interesting to give an example where where the conclusion of Theorem
  \ref{thm.chowhypertoric} fails for a non-quasi projective quotient
  of an open set in $(V \times V^*)$ on which $T$ acts properly cf. \cite[Remark 3.3]{deCMM:14}.
\end{remark}

\begin{proof}[Proof of Corollary \ref{cor.hypertoricorbifold}]
Again, by Theorem \ref{thm.orbifoldproducts} it suffices to show that
the pullback $(I\iota)^* \colon \CH^*(I\ix(A^{\pm}, \theta)) \to
\CH^*(I\iy(A,\theta))$ is an isomorphism of abelian groups. 
Unlike the proof of \ref{cor.cotangent} this not immediate
because $\ix(A^\pm, \theta)$ and $\iy(A,\theta)$ are not vector
bundles over a common base.

To prove this isomorphism we fix some notation.
Let $X(A^\pm, \theta) = (V \times V^*)^s$ 
and $Y(A,\theta) = \mu^{-1}(0) \cap X(A^\pm, \theta)$. Then
$\ix(A^\pm, \theta) = [X(A^\pm, \theta)/T]$
and $\iy(A^\pm, \theta) = [Y(A^\pm, \theta)/T]$. 

Since $T$ is diagonalizable and acts properly 
$I_T(X(A^\pm, \theta)) = \coprod_{\{g \in T | |g| < \infty\}} X(A^\pm, \theta)^g$
and 
$$I\ix(A^\pm, \theta) = [I_T(X(A^\pm, \theta))/T] = 
\coprod_{\{g \in T | |g| < \infty\}} [X(A^\pm, \theta)^g/T]$$
and a similar statement holds for $I\iy(A,\theta)$. (Note that all disjoint
sums are finite since if $T$ acts properly on a space $X$, $X^g =\emptyset$
for all but finitely many $g \in T$.)

Observe that if $V$ is a $T$-module and $g \in T$ then $(V^*)^g =
(V^g)^*$ so $V \times V^*)^g = V^g \times (V^g)^*$ as $T$-modules.
Also, if $\theta$ is a character then $$(V^g \times (V^*)^g)^s = (V^g
\times (V^*)^g) \cap (V \times V^*)^s$$ where stability is take with
respect to the character $\theta$.  Since the action of $T$ on $V$ is
diagonalized, the submodule $V^g$ is obtained by setting coordinates
$x_{i_1}, \ldots x_{i_k}$ to zero. Hence, $X(A^\pm,\theta)^g =
X(A_g^\pm, \theta)$ where $A_g$ is the matrix obtained from $A$ by
deleting the $i_1, \ldots i_k$th columns. (Note that $k$ and the
integers $i_1, \ldots , i_k$ depend on $g$.)  Hence $I\ix(A^\pm,
\theta) = \coprod_{\{g \in T | |g| < \infty \}} \ix(A_g^\pm, \theta)$.
A similar argument shows that 
$I\iy =  \coprod_{\{g \in T | |g| < \infty \}} \iy(A_g, \theta)$.
Hence by Theorem \ref{thm.chowhypertoric} the pullback
$I\iota^* \colon \CH^*(I\ix(A^\pm, \theta)) \to \CH^*(I\iy(A,\theta))$
is an isomorphism.
\end{proof}
We conclude with a conjecture.
\begin{conjecture} If $\tilde{X}$ is a canonical (hence toric)
  resolution of singularities of the Lawrence toric variety
  $X(A,\theta)$ and $\tilde{Y} = Y(A,\theta)\times_{X(A,\theta)}
  \tilde{X}$ then the pullback $\CH^*(\tilde{X}) \to \CH^*(\tilde{Y})$ is
  an isomorphism.
\end{conjecture}


\begin{thebibliography}{EJK2}

\bibitem[AV]{AbVi:02}
Dan Abramovich and Angelo Vistoli, {\em Compactifying the space of stable
  maps}, J. Amer. Math. Soc. \textbf{15} (2002), no.~1, 27--75 (electronic).

\bibitem[Art]{Art:69}
M.~Artin, {\em Algebraic approximation of structures over complete local
  rings}, Inst. Hautes \'Etudes Sci. Publ. Math. (1969), no.~36, 23--58.

\bibitem[BD]{BiDa:00}
Roger Bielawski and Andrew~S. Dancer, {\em The geometry and topology of toric
  hyperk\"ahler manifolds}, Comm. Anal. Geom. \textbf{8} (2000), no.~4,
  727--760.

\bibitem[Blo]{Blo:86}
Spencer Bloch, {\em Algebraic cycles and higher {$K$}-theory}, Adv. in Math.
  \textbf{61} (1986), no.~3, 267--304.

\bibitem[CR]{ChRu:02}
Weimin Chen and Yongbin Ruan, {\em Orbifold {G}romov-{W}itten theory},
  Orbifolds in mathematics and physics ({M}adison, {WI}, 2001), Contemp. Math.,
  vol. 310, Amer. Math. Soc., Providence, RI, 2002, pp.~25--85.

\bibitem[dMM]{deCMM:14}
M.~A. {de Cataldo}, L.~{Migliorini}, and M.~{Mustata}, {\em {The combinatorics
  and topology of proper toric maps}}, ArXiv e-prints (2014).

\bibitem[EG]{EdGr:98}
Dan Edidin and William Graham, {\em Equivariant intersection theory}, Invent.
  Math. \textbf{131} (1998), no.~3, 595--634.

\bibitem[EJK1]{EJK:10}
Dan Edidin, Tyler~J. Jarvis, and Takashi Kimura, {\em Logarithmic trace and
  orbifold products}, Duke Math. J. \textbf{153} (2010), no.~3, 427--473.

\bibitem[EJK2]{EJK:12a}
\bysame, {\em {A plethora of inertial products}}, Annals of K-theory, to appear
  (2014).

\bibitem[Ful]{Ful:84}
William Fulton, {\em Intersection theory}, Springer-Verlag, Berlin, 1984.

\bibitem[EGA4]{EGA4}
A.~Grothendieck and J.~Dieudonn\'e, {\em \'{E}lements de {G}\'eom\'etrie
  {A}lg\'ebrique {I}{V}. \'{E}tude locale des schemas et des morphismes de
  sch\'emas}, Inst. Hautes \'Etudes Sci. Publ. Math. No. \textbf{20, 24, 28,
  32} (1964, 1965, 1966, 1967).

\bibitem[HS]{HaSt:02}
Tam{\'a}s Hausel and Bernd Sturmfels, {\em Toric hyper{K}\"ahler varieties},
  Doc. Math. \textbf{7} (2002), 495--534 (electronic).

\bibitem[JKK]{JKK:07}
Tyler~J. Jarvis, Ralph Kaufmann, and Takashi Kimura, {\em Stringy {K}-theory
  and the {C}hern character}, Invent. Math. \textbf{168} (2007), no.~1, 23--81.

\bibitem[JT1]{JiTs:08a}
Yunfeng Jiang and Hsian-Hua Tseng, {\em Note on orbifold {C}how ring of
  semi-projective toric {D}eligne-{M}umford stacks}, Comm. Anal. Geom.
  \textbf{16} (2008), no.~1, 231--250.

\bibitem[JT2]{JiTs:08}
\bysame, {\em The orbifold {C}how ring of hypertoric {D}eligne-{M}umford
  stacks}, J. Reine Angew. Math. \textbf{619} (2008), 175--202.

\bibitem[Kol]{Kol:07}
J{\'a}nos Koll{\'a}r, {\em Lectures on resolution of singularities}, Annals of
  Mathematics Studies, vol. 166, Princeton University Press, Princeton, NJ,
  2007.

\bibitem[Kre]{Kre:99a}
Andrew Kresch, {\em Cycle groups for {A}rtin stacks}, Invent. Math.
  \textbf{138} (1999), no.~3, 495--536.

\bibitem[Lev]{Lev:01}
Marc Levine, {\em Techniques of localization in the theory of algebraic
  cycles}, J. Algebraic Geom. \textbf{10} (2001), no.~2, 299--363.

\bibitem[Low]{Low:15}
Daniel Lowengrub, {\em A cancellation theorem for {S}egre classes},
  arxiv:1503.01569.

\bibitem[Pro1]{Pro:11}
Nicholas Proudfoot, {\em All the {GIT} quotients at once}, Trans. Amer. Math.
  Soc. \textbf{363} (2011), no.~4, 1687--1698.

\bibitem[Pro2]{Pro:08}
Nicholas~J. Proudfoot, {\em A survey of hypertoric geometry and topology},
  Toric topology, Contemp. Math., vol. 460, Amer. Math. Soc., Providence, RI,
  2008, pp.~323--338.

\bibitem[Sum]{Sum:74}
Hideyasu Sumihiro, {\em Equivariant completion}, J. Math. Kyoto Univ.
  \textbf{14} (1974), 1--28.

\bibitem[Tho]{Tho:86d}
R.~W. Thomason, {\em Comparison of equivariant algebraic and topological
  {$K$}-theory}, Duke Math. J. \textbf{53} (1986), no.~3, 795--825.

\end{thebibliography}
\def\cprime{$'$} \def\cprime{$'$} \def\cprime{$'$}

\end{document}